\documentclass[a4paper,11pt]{article}
\usepackage[utf8x]{inputenc}

\usepackage{amsthm}
\usepackage{amssymb}
\usepackage{amsmath}
\usepackage{verbatim}
\usepackage{url}
\usepackage{mathtools}
\usepackage{tikz}
\usepackage{latexsym}
\usepackage{stmaryrd}
\usepackage{afterpage}
\usepackage{youngtab}
\usepackage{psfrag}
\usepackage[boxsize=.3 em]{ytableau}

\usetikzlibrary{chains}
\tikzset{node distance=2em, ch/.style={circle,draw,on chain,inner sep=2pt},chj/.style={ch,join},every path/.style={shorten >=4pt,shorten <=4pt},line width=1pt,baseline=-1ex}

\theoremstyle{plain}
\newtheorem{theorem}{Theorem}[section]
\newtheorem{proposition}[theorem]{Proposition}
\newtheorem{prop}[theorem]{Proposition}
\newtheorem{cor}[theorem]{Corollary}
\newtheorem{lemma}[theorem]{Lemma}
\newtheorem{corollary}[theorem]{Corollary}
\newtheorem{conjecture}[theorem]{Conjecture}
\newtheorem{cl}{Claim}
\theoremstyle{definition}
\newtheorem{defn}{Definition}[section] 

\theoremstyle{remark}
\newtheorem*{remark}{Remark}

\newcommand{\p}{{p}}
\newcommand{\w}{{w}}
\newcommand{\inv}{^{-1}}
\newcommand{\Y}{\mathbb{Y}}
\newcommand{\YP}{Y_P}
\newcommand{\YB}{Y_B}
\newcommand{\YPd}{Y_{P,d}}

\newcommand{\pr}{\mathrm{pr}}
\newcommand{\PSp}{\mathrm{PSp}}
\newcommand{\Spin}{\mathrm{Spin}}
\newcommand{\LG}{\mathrm{LG}}
\newcommand{\OG}{\mathrm{OG^{co}}}
\renewcommand{\P}{\mathbb{P}}
\newcommand{\so}{\mathfrak{so}}
\renewcommand{\Spin}{\mathrm{Spin}}
\newcommand{\Sym}{\mathrm{Sym}}
\newcommand{\id}{\mathrm{id}}
\newcommand{\RR}{\mathcal{R}}
\newcommand{\Vspin}{V_{\mathrm{Spin}}}
\newcommand{\PD}{\mathrm{PD}}
\newcommand{\Cl}{\mathrm{Cl}}
\newcommand{\D}{\mathcal{D}}
\newcommand{\N}{\mathcal{N}}
\newcommand{\End}{\mathrm{End}}
\newcommand{\C}{\mathbb{C}}
\newcommand{\Z}{\mathbb{Z}}

\newcommand{\alabel}[1]{%
  \({\mathrlap{#1}}\)
}
\newcommand{\mlabel}[1]{%
  \(#1\)
}
\let\dlabel=\alabel
\let\ulabel=\mlabel
\newcommand{\dnode}[2][chj]{%
\node[#1,label={below:\dlabel{#2}}] {};
}
\newcommand{\dnodea}[3][chj]{%
\dnode[#1,label={above:\ulabel{#2}}]{#3}
}
\newcommand{\dnodeanj}[2]{%
\dnodea[ch]{#1}{#2}
}
\newcommand{\dnodenj}[1]{%
\dnode[ch]{#1}
}
\newcommand{\dnodebr}[1]{%
\node[chj,label={below right:\dlabel{#1}}] {};
}

\newcommand{\dydots}{%
\node[chj,draw=none,inner sep=1pt] {\dots};
}

\newcommand{\QLeftarrow}{%
\begingroup
\tikz
\draw[shorten >=0pt,shorten <=0pt] (0,3pt) -- ++(-1em,0) (0,1pt) -- ++(-1em-1pt,0) (0,-1pt) -- ++(-1em-1pt,0) (0,-3pt) -- ++(-1em,0) (-1em+1pt,5pt) to[out=-105,in=45] (-1em-2pt,0) to[out=-45,in=105] (-1em+1pt,-5pt);
\endgroup
}

\title{A Landau-Ginzburg model for Lagrangian Grassmannians, Langlands duality and relations in quantum cohomology}
\author{C. Pech and K. Rietsch}

\begin{document}

\maketitle

\section{Introduction}

For a complex simple, simply connected algebraic group $G$ and parabolic subgroup $P$, the homogeneous space $G/P$ has a Landau-Ginzburg model defined by the second author \cite{rietsch}, which is a regular function on an affine subvariety of the Langlands dual group and is shown in \cite{rietsch} to recover the Peterson variety presentation \cite{peterson} of the quantum cohomology of $G/P$. In the case of type A Grassmannians R.~Marsh and the second author \cite{MarshRietsch} reformulated this Landau-Ginzburg model as a rational function on a Langlands dual Grassmannian, and used this formulation to prove a version of the mirror symmetry conjecture about flat sections of the A-model connection stated in \cite{BCFKvS}.

In this paper we formulate an LG-model $(\check X, W_t)$ for $G/P$ in the case of a Lagrangian Grassmannian in the spirit of the mirrors of the type $A$ Grassmannians, and prove that it is isomorphic to the LG-model from \cite{rietsch}. This LG model has some very interesting features, which are not visible in the type $A$ case, to do with the non-triviality of Langlands duality. We also formulate an explicit conjecture relating our superpotential
with the quantum differential equations of $LG(m)$. Finally, our expression for $W_t$ also leads us to conjecture new formulas in the quantum Schubert calculus of $LG(m)$.

To give an idea of our result, which is very explicit, we give the first two interesting examples here. Note that the Schubert basis of $H^*(LG(m))$ is indexed by strict partitions $\lambda$ fitting in an $m\times m $ box
and can be identified with coordinates $\p_\lambda$ on the Grassmannian $\OG(m+1,2m+1)$ of $(m+1)$-dimensional co-isotropic subspaces of $\C^{2m+1}$ endowed with a non-degenerate quadratic form. Note that $\OG(m+1,2m+1)$ is canonically isomorphic to the maximal orthogonal Grassmannian $\mathrm{OG}(m,2n+1)$. Moreover, it is related to $X$ by Langlands duality. The goal of this paper is to give an explicit description of a Landau-Ginzburg model for $LG(m)$ as a rational function on $\OG(m+1,2m+1)$. As an example, for $LG(2)$ our Landau-Ginzburg model is the rational function on $\OG(3,5)$ given by
$$
W_t=\frac{\p_{\ydiagram{1}}}{\p_{\emptyset}}+\frac{\p_{\ydiagram{2}}^2}{\p_{\ydiagram{1}}\p_{\ydiagram{2}}-\p_{\emptyset}\p_{\ydiagram{2,1}}}+
e^t\frac{\p_{\ydiagram{1}}}{\p_{\ydiagram{2,1}}}.
$$
For $LG(3)$ we obtain the rational function on $\OG(4,7)$,
$$
W_t=\frac{\p_{\ydiagram{1}}}{\p_{\emptyset}}+
\frac{\p_{\ydiagram{2}}\p_{\ydiagram{3}}-\p_{\emptyset}\p_{\ydiagram{3,2}}}
{\p_{\ydiagram{1}}\p_{\ydiagram{3}}-\p_{\emptyset}\p_{\ydiagram{3,1}}}+
\frac{\p_{\ydiagram{3,1}}\p_{\ydiagram{3,2}}-\p_{\ydiagram{3}}\p_{\ydiagram{3,2,1}}}
{\p_{\ydiagram{2,1}}\p_{\ydiagram{3,2}}-\p_{\ydiagram{2}}\p_{\ydiagram{3,2,1}}}+
e^t\frac{\p_{\ydiagram{2,1}}}{\p_{\ydiagram{3,2,1}}}.
$$
We generalise these formulas and prove that they agree with the superpotential from \cite{rietsch} after suitable identifications.

Notice how  the above formulas have $3, 4$ summands, these numbers being the index of $X=LG(2)$, $LG(3)$, respectively. Indeed this comes from the fact that in all of the cases $W_t$ represents the anti-canonical class of $X$ in a natural sense (in the Jacobi ring for example), and each summand represents a hyperplane class. On the other hand, because $W_t$ wants to be regular in the complement of an anti-canonical divisor, the degrees of the denominators in $W_t$ should add up to the index of $\check X$. That is,  in the above two cases to $4$ and $6$, these being the index of $\OG(3,5)$ and $\OG(4,7)$, respectively. This is exactly what is achieved by the quadratic terms in the $LG(m)$ cases, with $1+2+1=4$, and $1+2+2+1=6$ (and so forth, in our general formula).

For usual Grassmannians $X$ and $\check X$ are isomorphic so have the same index. 
Therefore numerators and denominators in $W_t$ are allowed to be sections of $\mathcal O(1)$.
This leads to the formulas in \cite{MarshRietsch} looking more compact.

\section{Background}

In \cite{rietsch}, the second author gave a Lie-theoretic construction of a Landau-Ginzburg model of any complete homogeneous space $X$  of a simple complex algebraic group. The LG-model $(\check X^\circ, W)$ is set in the world of the Langlands dual group. 

\subsection{Notation}

Let $X$ be a complete homogeneous space for a simple complex algebraic group. For the purposes of this paper we will denote the group acting on $X$ by $G^\vee$ and assume that $G^\vee$ is simply connected,  and we will denote its Langlands dual group by $G$, which is therefore an adjoint group.
For $G^\vee$ we may fix Chevalley generators $(e_i^\vee)_{1 \leq i \leq m}$ and $(f_i^\vee)_{1 \leq i \leq m}$ and correspondingly Borel
subgroups $B^\vee_+=T^\vee U^\vee_+$ and $B^\vee_-=T^\vee U^\vee_-$. We may
assume that $X=G^\vee/P^\vee$ for a parabolic subgroup $P^\vee$ which contains $B^\vee_+$. The parabolic $P^\vee$ is determined by a choice of subset of the $(f_i^\vee)_{1 \leq i \leq m}$. This set also determines a parabolic subgroup $P$ of $G$, where we also have the analogous Borel subgroups $B_+=T U_+$ and $B_-=T U_-$ and Chevalley generators $(e_i)_{1 \leq i \leq m}$ and $(f_i)_{1 \leq i \leq m}$. Let $\Pi=\{\alpha_i\ |\ i\in I\}$ denote the set of simple roots. The set of all roots is $R=R^+ \sqcup R^-$, where $R^+$ is the subset of positive roots and $R^-$ the subset of negative roots. 

Denote by $W$ the Weyl group of $G$ (canonically identified with the Weyl group of $G^\vee$), and let $W_P$ be the Weyl group of the parabolic subgroup~$P$. Let $T^{W_P}$ be the $W_P$-fixed sub-torus. 
If $\alpha$ is a positive root, we denote by $s_\alpha \in W$ the associated reflection.  Let $R_P^+$ be the set of all positive roots $\alpha$ such that $s_\alpha \in W_P$, $\Pi_P$ be the set of simple roots in $R_P^+$, and $\Pi^P=\Pi\setminus\Pi_P$.  
When $\alpha=\alpha_i$ is a simple root, we set $s_i := s_{\alpha_i}$. Moreover, we denote the length of $w\in W$ by $\ell(w)$. It is equal to the minimum number of simple reflections whose product is $w$. We also let $w_0$ and  $w_P$, be the longest elements in $W$ and $W_P$, respectively, and define $W^P$ to be the set of minimal length coset representatives for $W/W_P$. The minimal length coset representative for $w_0$ is denoted by $w^P$, so that $w_0=w^P w_P$. Let $\dot{w}$ denote a representative of $w \in W$ in $G$.

Using the exponential map we may  think of $U_+$ and $U_-$ as being embedded in the completed universal enveloping algebra $\hat{\mathcal U}_+$, respectively $\hat{\mathcal U}_+$. Accordingly $e_i^*(u)$  will denote the coefficient of $e_i$ in $u\in U_+$ after this embedding, and analogously for $f_i^*$ and $\bar u\in U_-$.

\subsection{Quantum cohomology of $G/P$}

The quantum cohomology ring of a smooth complex projective variety $X$ is a deformation of its cohomology ring. While the cohomology ring of $X$ encodes the way its subvarieties intersect each other, the quantum cohomology ring encodes the way they are connected by rational curves. The structure constants of the (small) quantum cohomology ring are called Gromov-Witten invariants. When $X=G/P$ is homogeneous, Gromov-Witten invariants count the number of rational curves of given degree intersecting three given Schubert varieties of $X$.

The quantum cohomology rings of a full flag variety was first described by Givental and B.~Kim \cite{GiventalKim,Kim}, who related it to a degenerate leaf of the Toda lattice of the Langlands dual group. Soon after, Dale Peterson came up with a new point of view in which all of the quantum cohomology rings of complete homogeneous spaces for one group are encoded in terms of strata of one remarkable subvariety of the Langlands dual full flag variety. This so-called {\it Peterson variety} $\Y$ is defined as follows. In our conventions Peterson's variety $\Y$ encoding the quantum cohomology rings of $G^\vee$-homogeneous spaces is a subvariety of $G/B_-$. Denote by $\mathfrak n_-$ the Lie algebra of $U_-$, and by $[\mathfrak n_-,\mathfrak n_-]$ its commutator subalgebra. The annihilator in $\mathfrak g^*$ of a subspace $\mathfrak l$ of $\mathfrak g$ is denoted by $\mathfrak l^\perp$. Consider the coadjoint action of $G$ on $\mathfrak g^*$ and the  `principal nilpotent' element $F=\sum e_i^*$ in $\mathfrak g^*$. Then
\[
\Y:=\{gB_-\ |\ g\inv \cdot F\in [\mathfrak n_-,\mathfrak n_-]^\perp\}.
 \]
First note that this variety has an open stratum $\YB=\Y\cap (B_+ B_-/B_-)$ which is isomorphic to the degenerate leaf of the Toda lattice for $G$ via the map $\YB\hookrightarrow \mathfrak g^*$ defined by $u_+B_-\mapsto u_+^{-1} \cdot F$. By Peterson's theory, the quantum cohomology rings for all other $G^\vee/P^\vee$ are described by the coordinate rings of the smaller strata $\YP=\Y\cap (B_+ \dot w_PB_-/B_-)$, where we take the intersection in the possibly non-reduced sense.

\begin{theorem}[Peterson]
 The quantum cohomology of $G^\vee/P^\vee$ is isomorphic to the coordinate ring $\C[\YP]$ of the stratum $\YP$ of the Peterson variety $\Y$.
\end{theorem}

In \cite{FW}, Fulton and Woodward proved a quantum Chevalley formula for $X=G/P$, i.e. a formula giving the product of an arbitrary Schubert class by any Schubert class associated to a Schubert divisor. Here we state this formula, which we will refer to in Section \ref{s:relations}. Note that for $P=B$, the formula is a result of Peterson \cite{peterson}.

If $s_i$ is a simple reflection, we denote by $\Gamma_i \in H_2(X,\Z)$ the associated dimension $1$ Schubert cycle, and we define, for $\alpha \in R^+ \setminus R_P^+$ :
\[
 d(\alpha) := \sum_{i=1}^m \alpha^\vee(\omega_i) \Gamma_i.
\]
Now set $q^{d(\alpha)} := \prod_{i=1}^m q_i^{\alpha^\vee(\omega_i)}$, where $q_i$ is the quantum parameter associated to $\Gamma_i$. Finally, for $\alpha \in R^+ \setminus R_P^+$, we define $n_\alpha := \int_{\Gamma_\alpha} c_1(TX)$, where $\Gamma_\alpha \in H_2(X,\Z)$ is the dimension $1$ cycle associated to the reflection $s_\alpha$ (it is a linear combination of the $\Gamma_i$).

\begin{theorem}[\cite{FW}]
\label{th:quantChevalley}
 For $1 \leq i \leq m$ and $w \in W^P$ we have
 \[
  \sigma_{s_i} \star \sigma_w = \sum_\alpha \alpha^\vee(\omega_i) \sigma_{w s_\alpha} + \sum_{\alpha} q^{d(\alpha)} \alpha^\vee(\omega_i) \sigma_{w s_\alpha},
 \]
 where the first sum is over roots $\alpha \in R^+ \setminus R_P^+$ such that $l(w s_\alpha)=l(w)+1$, and the second sum over roots $\alpha \in R^+ \setminus R_P^+$ such that $l(w s_\alpha)=l(w)+1-n_\alpha$.
\end{theorem}

\subsection{The Lie-theoretic LG model construction}

We recall how the mirror Landau-Ginzburg models are defined in \cite{rietsch}. Let us fix a parabolic $P$. We consider the open Richardson variety $\RR := R_{w_P,w_0} \subset G / B_-$, namely 
\[ 
\RR:=R_{w_P,w_0} = (B_+\dot w_P B_-  \cap B_-\dot w_0 B_- )/ B_-.
\]
Instead of the whole stratum $\YP$ of the Peterson variety the LG-model is related to the open dense subset $\YP^*:=\Y\cap\RR$, whose coordinate ring in Peterson's theory encodes the quantum cohomology ring $qH^*(G^\vee/P^\vee)$ with quantum parameters made invertible. We note that in this setting if  $g=u_1 d \dot w_P \bar u_2=b_-\dot w_0$ represents an element $gB_-\in\RR$ lying in  $\YP^*$, then the values of the functions on $\YP^*$ corresponding to the quantum parameters are just the values $\alpha_j(d)$ for the simple roots $\alpha_j\in \Pi^P$. Indeed, fixing $d\in T^{W_P}$ determines a finite subscheme of $\YP^*= \Y\cap \RR$ which we denote by $\YPd^*=\YP^*\times_{T^{W_P}}\{d\}$ and for which the non-reduced coordinate ring $\C[\YPd^*]$ becomes identified with the quantum cohomology ring of $G^\vee/P^\vee$ with quantum parameters fixed to the values $\alpha_j(d)$ in Peterson's theory.
 
Now let us define
$$
Z=Z_{G^\vee/P^\vee} := B_- \dot w_0 \cap U_+ T^{W_P} \dot w_P U_- .
$$ 
There is an isomorphism
\[
 \begin{array}{ccc}
  Z & \rightarrow & \RR \times T^{W_P}, \\
 g= u_1 d \dot w_P  \bar u_2 = b_-\dot w_0& \mapsto &(gB_-,d).
 \end{array}
\]
Observe that $ gB_-=b_- \dot w_0 B_-=u_1 \dot w_P B_-$. Note that our conventions differ from \cite{rietsch} in that we have translated the original definition of the variety $Z$ by $\dot w_0$. The mirror superpotential to $X=G^\vee/P^\vee $ is now defined to be the regular function $\mathcal F: Z \to \C$ defined by

\begin{equation}\label{e:oldLGnoh}
      \mathcal F(u_1 d \dot w_P  \bar u_2 ) = \sum_{i=1}^m e_i^* (u_1) + \sum_{i=1}^m f_i^* (\bar u_2).
\end{equation}
Although $u_1$ and $\bar u_2$ are not uniquely determined for $g\in Z$, the function $\mathcal F$ is well-defined, as was shown in \cite{rietsch}. Actually, there is another small difference with \cite{rietsch}, in that in \cite{rietsch} the group on the mirror side is assumed adjoint, whereas here we have assumed $G$ to be simply connected. However we could have carried out the above definitions  for $G/\mathrm{Center}(G)$, and in the following it will not matter. 

The superpotential $\mathcal F$ may also be interpreted as a family of functions $\mathcal F_h:\RR\to\C$ depending holomorphically  on a parameter $h\in \mathfrak h^{W_P}$, by setting
\begin{equation}\label{e:oldLG}
\   \mathcal F_h 
       (u_1 \dot w_P B_-) =   \sum_{i=1}^m e_i^* (u_1) + \sum_{i=1}^m f_i^* (\bar u_2)
\end{equation}
where $u_1\in U_+$ and $u_1\dot w_PB_-\in\RR$, and where $\bar u_2\in U_-$ is related to $u_1$ by $u_1 e^h \dot w_P \bar u_2\in Z$. Equivalently the relationship between $u_1$ and $\bar u_2$ can be expressed as 
\[ 
 \bar u_2 \cdot B_+= e^{-h}\dot w_P\inv u_1\inv \cdot  B_-. 
\]  
where $g\cdot B$ denotes the conjugation action of $g\in G$ on a Borel subgroup $B$. 

The main result in \cite{rietsch} describes the critical point scheme of $\mathcal F_h$ as subscheme of $\RR$ lying inside the Peterson variety. We denote by $Y_{P, e^h}^*$ the (non-reduced) fiber over $e^h$ of the Peterson variety, namely
\[
 Y_{P, e^h}^*=\YP^*\times_{T^{W_P}}\{e^h\}.
\]

\begin{theorem}
[\cite{rietsch}]
The critical point scheme of $\mathcal F_h$ agrees with $Y_{P, e^h}^*$.
\end{theorem}

Putting this together with Peterson's presentation this result can be interpreted as follows. Suppose $h\in\mathfrak h^{W_P}$ represents a Kaehler class $[\omega_h]$ under the identification $\mathfrak h^{W_P}=H^2(G^\vee/P^\vee)$.
\begin{cor}
\label{theo:LGmodel-G/P} 
 The Jacobi ring $\C[Z_h]/(\partial \mathcal F_h)$ of  $\mathcal F_h: Z_h \rightarrow  \C$ is isomorphic to the quantum  cohomology ring of the Kaehler manifold $(G^\vee/P^\vee,[\omega_h])$ in its  presentation due to Dale Peterson\cite{peterson}. 
\end{cor}
 
In \cite{MarshRietsch}, R. Marsh and the second author gave an expression of the Landau-Ginzburg model of the Grassmannian in terms of Pl\"ucker coordinates and then described the A-model connection. Here we will express the Landau-Ginzburg model of the Lagrangian Grassmannian in terms of \emph{generalized Pl\"ucker coordinates}, i.e the coordinates of its minimal embedding $\OG(m+1,2m+1) \hookrightarrow \P(V_{\Spin})$.

\section{The Lagrangian Grassmannian and its LG model}

Let $G^\vee=\PSp_{2m}(\C)$, the adjoint group of type  $C_m$, with Dynkin diagram 
\[
 \begin{tikzpicture}[start chain]
 \dnode{1}
 \dnode{2}
 \dydots
 \dnode{}
 \dnodenj{m}
 \path (chain-4) -- node[anchor=mid] {\(\Leftarrow\)} (chain-5);
 \end{tikzpicture}\quad .
\]
Let $P^\vee:=P_{\omega_m^\vee}$ be the parabolic subgroup associated to the $m$-th fundamental weight $\omega_m^\vee$ of $G^\vee$. The quotient $G^\vee/P^\vee$ is the homogeneous space called the \emph{Lagrangian Grassmannian}, which parametrizes Lagrangian subspaces in $\C^{2m}$. It is also denoted by $X=\LG(m)$ and will play the role of the $A$-model for us. 

Now the Langlands dual group $G$ is the simply connected group of type $B_m$, namely the spin group $\Spin_{2m+1}(\C)$, 
\[
 \begin{tikzpicture}[start chain]
 \dnode{1}
 \dnode{2}
 \dydots
 \dnode{\ }
 \dnodenj{m}
 \path (chain-4) -- node[anchor=mid] {\(\Rightarrow\)} (chain-5);
 \end{tikzpicture} \quad .
\]
The parabolic subgroup of $\Spin_{2m+1}(\C)$ associated to the $m$-th fundamental weight is denoted $P=P_{\omega_m}$. In this (B-model) setting we  consider the quotient from the left $\check{X} := P \backslash G $. This quotient may be interpreted as the co-isotropic Grassmannian $ \OG(m+1,2m+1)$ in a vector space of row vectors. We consider it in its minimal embedding, namely the homogeneous space $\check{X} := P \backslash G$ is embedded in $\mathbb P(V_{\omega_m}^*)$ as right $G$-orbit of the highest weight vector $w_\emptyset^*$. We will express the mirror Landau-Ginzburg model to $LG(m)$  as a rational function on the orthogonal Grassmannian~$\check X$ in the homogeneous coordinates of this embedding.

\begin{remark} Note that the Lagrangian Grassmannian $X=LG(m)$ is a cominiscule homogeneous space of type $C_m$, and therefore its cohomology appears in geometric Satake correspondence \cite{lusztig,MV,ginzburg} as 
\[
 H^*(LG(m))=IH^*(\overline{Gr_G^{\omega_m}})=V_{\omega_m}^{\Spin_{2m+1}}.
\]
In other words it is canonically identified with the unique miniscule representation, the spin representation $V_{\omega_m}^{\Spin_{2m+1}}$ also denoted $V_{\Spin}$, of the Langlands dual group, $G=\Spin(2m+1)$. Therefore, essentially tautologically, $\mathbb P(V_{\Spin}^*)$ has homogeneous coordinates given by the Schubert basis of $H^*(LG(m))$. 
\end{remark}

\subsection{Notations and conventions}\label{s:notations}

Let $v_1,\dotsc, v_{2m+1}$ be the standard basis of $V=\C^{2m+1}$, and fix the symmetric non-degenerate bilinear form 
\[
\left<v_i,v_{2m+2-j}\right>=2\Phi(v_i,v_{2m+2-j})=(-1)^{m+1-i}\delta_{i,j}.
\]
We may also use the notation $\bar v_{j}=v_{2m+2-j}$ (with decreasing $j$) for the basis elements $v_{m+2},\dotsc, v_{2m+1}$ and set $\epsilon(i):=(-1)^{m+1-i}$ so that $\Phi(v_i,\bar v_i)=\epsilon(i)$. The subspace of $V$ spanned by the first $m$ basis vectors $v_1,\dotsc, v_m$ is maximal isotropic and denoted by $W$. 

We let $G=\Spin(V)=\Spin(V,\Phi)$, which is the universal covering group of $SO(V,\Phi)$. The Lie algebra of $G=\Spin(V)$ is therefore $\so(V)=\so(V,\Phi)$  which we view as lying in $\mathfrak gl(V)$. We have explicit Chevalley generators $e_i,f_i$ given by 
\begin{eqnarray*}
 e_i&=&E_{i,i+1}+E_{2m+1-i,2m+2-i}\quad \text{for $i=1,\dotsc, m-1$,}\\
 e_m&=&\sqrt{2}E_{m,m+1}+\sqrt{2}E_{m+1,m+2},\\
 f_i&=& e_i^T \quad \text{for $i=1,\dotsc, m$.}
\end{eqnarray*}
Here $E_{i,j}$ is the $(2m+1)\times (2m+1)$-matrix with $(i,j)$-entry $1$ and all other entries $0$. We also define the corresponding group homomorphisms $x_i:\C\to G$ and $y_i:\C\to G$, namely $x_i(a):=\exp(ae_i)$ and $y_i(a):=\exp(a f_i)$.

Next we introduce notations for the  Clifford algebra $\Cl(V)$ and the Spin representation $V_{\Spin}$, see also \cite{Varadarajan} whose conventions we follow for the most part. The Clifford algebra $\Cl(V)$ is the algebra quotient of the tensor algebra $T(V)$ by the ideal generated by the expressions
\[
 v\otimes v' + v'\otimes v - 2\Phi(v,v').
\]
So it is the algebra with generators $ v_{m+1}$ and $v_i,\bar v_i$ for $i=1,\dotsc, m$, with relations
\[
 v_i\bar v_i+\bar v_i v_i=\epsilon(i),\quad v_{m+1}^2=\frac 12,
\]
and where all other generators anti-commute. The Clifford algebra is $\Z / 2\Z$-graded, as the relations are in even degrees only, and the even part of $\Cl(V)$ is denoted by $\Cl^+(V)$.  

Since $\Spin(V)$ acts on $V$, it acts on  $\bigwedge^\bullet V$, and because it preserves the bilinear form $\Phi$, it also acts on $\Cl(V)$. The anti-symmetrization map 
\begin{eqnarray*}
 {\bigwedge}^k V&\to &\Cl(V)\\
 v_{i_1}\wedge\dotsc \wedge v_{i_k} &\mapsto & \frac{1}{k!}(\sum_{\sigma\in S_k}v_{i_{\sigma(1)}} v_{i_{\sigma(2)}}\cdots
 v_{i_{\sigma(k)}}).
\end{eqnarray*}
is an embedding of representations, and we will usually identify elements of ${\bigwedge}^k V$ with their images, as we are mainly interested in the algebra structure of the Clifford algebra. The representation $ {\bigwedge}^2 V$ is isomorphic to the adjoint representation. Moreover the image of $ {\bigwedge}^2 V$ in $\Cl(V)$ is indeed a Lie algebra under the commutator Lie bracket of $\Cl(V)$, and it is isomorphic to $\so(V)$ as such. In particular our generators $e_i, f_i$ can be identified with elements of $\bigwedge^2V$ and their images in $\Cl(V)$. Under this identification they are given by
\begin{eqnarray*}
  e_i=\epsilon(i+1) v_{i}\wedge \bar v_{i+1}&=&\epsilon(i+1) v_{i} \bar v_{i+1} \text{ for $i=1,\dotsc, m-1$}\\ 
 e_m={\sqrt 2}v_m\wedge v_{m+1}&=&{\sqrt 2}v_m v_{m+1},\\
  f_i=\epsilon(i) v_{i+1}\wedge \bar v_{i}&=&\epsilon(i)  v_{i+1}\bar v_{i} \text{ for $i=1,\dotsc, m-1$},\\ 
 f_m={\sqrt 2} \bar v_{m}\wedge v_{m+1}&=&{\sqrt 2}\bar v_m v_{m+1}. 
\end{eqnarray*}
Putting all of the anti-symmetrization maps together gives an isomorphism of $\so(V)$-modules
\[
\bigwedge{}^\bullet\, V\longrightarrow \Cl(V).
\]
Moreover the even wedge powers map to the even part $\Cl^+(V)$ of the Clifford algebra and odd ones to the odd part, $\Cl^-(V)$. Therefore we have two isomorphisms of $\so(V)$-modules
\begin{eqnarray}
 \label{e:evenPartsIso}
 \alpha_+:\bigwedge{}^{\text {even}} \ V &\longrightarrow& \Cl^+(V),\\
 \label{e:oddPartsIso}
 \alpha_-: \bigwedge{}^{\text {odd}} \ V &\longrightarrow& \Cl^-(V).
\end{eqnarray}

The Spin representation, as a vector space, is $V_{\Spin}=\bigwedge^\bullet W$. Its standard basis elements are the elements $\w_I:=v_{i_1}\wedge\dotsc \wedge v_{i_k}$ with $i_1<i_2<\cdots <i_k$, where $I=\{i_1,\dotsc, i_k\}$ is any subset of $\{1,\dotsc, m\}$. We sometimes write $[v_{i_1}\wedge\dotsc \wedge v_{i_k}]$ instead of $v_{i_1}\wedge\dotsc \wedge v_{i_k}$ when we mean the element of $V_{\Spin}$. Note that if  $I=\emptyset$ then $w_\emptyset=[1]$.  

The subsets $I$ are also in one-to-one correspondence with strict partitions $\lambda$ contained in an $m \times m$ square, by sending the empty set to the empty partition, and 
\[
 I=\{i_1,\dotsc, i_k\}\quad \mapsto\quad \lambda=(m+1-i_1,m+1-i_2,\dotsc, m+1-i_k).
\]
In this correspondence the $k$-row partitions correspond to the basis elements in the $k$-th graded component, $\bigwedge^k W$, of $V_\Spin$. We may denote $w_I$ also by $w_\lambda$.  If $\lambda$ is a strict partition contained in an $m\times m$ rectangle, then we denote by $|\lambda|$ the sum of all its parts and by $\PD(\lambda)$ the Poincar\'e dual partition. 

The Spin representation of $\so(V)$ extends to a representation of the Clifford algebra, which can be defined on generators by 
\[
 v_i\cdot w_{I}=v_i\wedge w_{I}, \quad v_{m+1}\cdot w_{I}=\frac{(-1)^{|I|}}{\sqrt 2} w_{I}, \quad  \bar v_j\cdot w_{I}=i_{\bar v_j}(w_I),
\] 
where $i_{\bar v_i}$ is the insertion operator on $\bigwedge^\bullet W$, for $\bar v_i$ identified with the linear form $2\Phi(\bar v_i, \quad )$ on $W$. 

We recall the important fact that the even subalgebra  $\Cl^+(V)$ of the Clifford algebra is isomorphic to $\End(V_{\Spin})$ via the action just defined. Combined with the map \eqref{e:evenPartsIso} we obtain an isomorphism of $\so(V)$-modules
\begin{equation}\label{e:EvenToEnd}
 \kappa_+:\bigwedge{}^{\text {even}}\ V\longrightarrow  \End(V_\Spin).
\end{equation}
Moreover there is also an isomorphism of $\mathfrak{so}(V)$-modules,
\begin{equation}\label{e:OddToEnd}
 \kappa_-:\bigwedge{}^{\text {odd}}\ V\longrightarrow  \End(V_\Spin)
\end{equation}
given by antisymmetrization, $\alpha_-: \bigwedge{}^{\text {odd}}\to\Cl^-(V)$ followed by the action of $\Cl^-(V)$ on $V_{\Spin}$.

The standard basis $\{w_I\}$ of $V_{\Spin}$ defined above is also precisely the integral weight basis obtained by successively applying generators $e_i$ to the lowest weight vector $\w_\emptyset=[1]$, and it agrees with the $MV$-basis of $V_\Spin$, which in this case is one and the same as the Schubert basis of $H^*(LG(m))$. We will use the notation $\sigma_\lambda$ for the Schubert basis element naturally identified with $\w_\lambda$.  

The generalized Pl\"ucker coordinates on our $\OG(m+1,2m+1)=P\backslash G$ are the sections of $\mathcal O(1)$ in the embedding $P\backslash G\hookrightarrow \P(V_\Spin^*)$ which are given by the basis elements $w_\lambda$ of $V_{\Spin}$ described above. Explicitly, we define   
\[
 \p_\lambda(g) := \langle \w_{\emptyset}^*\cdot g , \w_\lambda \rangle= \w_{\emptyset}^* (g \cdot \w_\lambda) ,
\]
where $\w_\emptyset^*$ is the dual basis vector to $\w_\emptyset$, which is therefore a highest weight vector of $V_{\omega_m}^*$, and where $\w_\lambda$ is as above. We may think of an element $Pg\in \OG(m+1,V^*)=P\backslash G$ as specified by its homogeneous coordinates $(\p_{\lambda_1}(g)\colon \p_{\lambda_2}(g)\colon \dotsc\colon \p_{\lambda_{2^m}}(g))$, where $\lambda_1,\dotsc, \lambda_{2^m}$ are the strict partitions in $m\times m$ in some ordering.

To summarize, associated to strict partitions $\lambda \subset m \times m$, or equivalently subsets $I$ of $\{1,\dotsc, m\}$, we have elements
\[
 \sigma_\lambda\in H^*(LG(m)), \quad w_\lambda\in V_{\Spin}, \quad\text{and}\quad \p_\lambda\in \Gamma[\mathcal O_{\OG(m+1,V^*)}(1)], 
\]
all canonically identified. We may also denote them by $\sigma_I,\w_I$ and $\p_I$, respectively.

For a later section we will also require an explicit isomorphism $V\cong V^*$. Since $V$ has on it a quadratic form, we have that $V\cong V^*$ by $v\mapsto \left<v,\ \right>$ and $V^*$ has basis $v_1^*,\dotsc, v_{m+1}^*,  v_{m+2}^*,\dotsc, v_{2m+1}^*.$
Under the isomorphism with $V$ this basis corresponds to 
\begin{equation*}
\begin{matrix}v_1^*=\epsilon(1)\bar v_1 &&v_{2m+1}^*=\bar v_1^*=\epsilon(1) v_1\\
 v_2^*=\epsilon(2)\bar v_2&&v_{2m}^*=\bar v_2^*=\epsilon(2) v_2\\
  \vdots  &&  \vdots \\
  v_m^*=-\bar v_m && v_{m+2}^*=\bar v_{m}^*=- v_m  \\
   &  v_{m+1}^*=v_{m+1}.&
\end{matrix}
\end{equation*}

\subsection{Definition of $W_t$}

We will now explain our formula for $W_t:\OG(m+1,V^*)\to \C$ in terms of the coordinates $\p_\lambda$. Here are some particular partitions which will play an important role. Let $\rho_l := (l,l-1,\dots,2,1)$ be the length $l$ staircase partition and let $\mu_l := (m,m-1,\dots,m+1-l)$ be the maximal strict partition with $l$ lines contained in an $m\times m$ rectangle. For $\rho_l$ with $l<m$ there is a unique strict partition obtained by adding a single box to the Young diagram. It is obtained by adding one box to the first line, and we denote it by $\rho_{l,+}$. 
If $J$ is any subset of $\left\{1,\dots,l\right\}$, we denote by $\rho_l^{J}$ the partition obtained after removing for every $j\in  J$ the $j$-th line from the Young diagram of $\rho_l$ (and similarly for $\rho_{l,+}^{J}$). On the other hand we denote by $\mu_l^J$ the partition obtained by adding for each $j\in J$ a row of $l+1-j$ boxes to the bottom of $\mu_l$. Similarly, $\mu_{l,+}^{J}$ is obtained by adding for each $j\in J$ a row of $l+1-j+\delta_{j,1}$ boxes to the bottom of $\mu_l$. If the resulting Young diagram does not give a strict partition, then we set $\mu_l^{J}=0$, respectively $\mu_{l,+}^{J}=0$. Finally, set  $s(J):= \sum_{j \in J} j$ for any subset $J$ of $\{1,\dotsc, m\}$. 

Using the above notations, we define $W_t:\OG(m+1,V^*)\to \C$ by
\begin{equation}\label{e:Wt}
 W_t := \frac{\p_{\rho_{0,+}}}{\p_{\rho_0}} + \sum_{l=1}^{m-1} \frac{\underset{J\subset\left\{1,\dots,l\right\}}{\sum}(-1)^{s(J)}\p_{\rho_{l,+}^{J}}\p_{\mu_{l,+}^{J}}}{\underset{J\subset\left\{1,\dots,l\right\}}{\sum}(-1)^{s(J)}\p_{\rho_l^{J}}\p_{\mu_l^J}} + e^t \, \frac{\p_{\rho_{m-1}}}{\p_{\rho_m}}. 
\end{equation}
This is a rational function on $\check X=\OG(m+1,V^*)$. Inside $\check X$ the denominators in $W_t$ give rise to divisors
\[
 D_0:=\{\p_{\emptyset}=0\} ,\quad
 D_m:=\{\p_{\rho_m}=0\}
\]
and
\[
 D_l:= \left\{ \underset{J\subset\left\{1,\dots,l\right\}}{\sum}(-1)^{s(J)}\p_{\rho_l^{J}}\p_{\mu_l^J}=0\right \},\quad \text{where $l=1,\dotsc, m-1$.}
\]
Then 
\[
 D:=D_0+D_1+\dotsc + D_{m-1}+ D_m
\]
is an anti-canonical divisor. Indeed, the index of $\check X=\OG(m+1,V^*)$ is $2m$. We define $\check X^\circ:=\check X\setminus D$. The restriction of our rational function $W_t$ to $\check X^\circ$ is regular, and is again denoted $W_t$.

We would like to compare $W_t:\check X^\circ\to \C$ with the known super-potential of $X=LG(m)$ defined as a special case of \eqref{e:oldLG}. Explicitly recall that $LG(m)=G^\vee/P^\vee$ for $ G^\vee=PSp(2m) $ with $P^\vee$ the parabolic corresponding to the $m$-th node of the Dynkin diagram $C_m$. The function $\mathcal F_h$ for $h\in \mathfrak h^{W_P}$ is  therefore defined on the open Richardson variety $\RR=B_+w_P B_-\cap B_- \dot w_0 B_-/B_-$ inside the full flag variety of $G=\Spin(V)$, where $P$ is the parabolic corresponding to the $m$-th node of $B_m$. So we would like to relate our variety  $\check X=P\backslash G=\OG(m+1,V^*)$, or rather its open part $\check X^\circ$, with this open Richardson variety. 
The parameter $t$ in $W_t$  and the $h\in \mathfrak h^{W_P}$ appearing in $\mathcal F_h$ should be thought of as equivalent,  by the relation $h=t \omega_m^\vee$. 

For fixed parameter $t$ we define the following maps
\begin{eqnarray*}\label{e:DomainComparison}
 \OG(m+1,V^*)=P\backslash G\ \overset{\Psi_L}  \longleftarrow   &B_-\dot w_0\cap U_+ e^{t\omega_m^\vee}\dot w_P\dot U_- &  \overset{\Psi_R}{\longrightarrow}\  \RR,\\
 P g\ \longleftarrow\ & g & \rightarrow g B_-.
\end{eqnarray*}
given by taking left and right cosets, respectively. Note that $g=b_-\dot w_0$ in our previous notation and factorizes as
\[
 g=u_1e^{t\omega_m^\vee}\dot w_P\bar u_2,
\]
Moreover $\Psi_R$ is an isomorphism, so we have $\Psi:=\Psi_L\circ\Psi_R\inv:\RR\to \OG(m+1,V^*)$.
 Our main goal here is to prove the theorem. 

\begin{theorem}
\label{theo:W}
 Let $X=LG(m)$ and $t\in \C$. The rational function $W_t$ on  $\OG(m+1,V^*)$ defined in \eqref{e:Wt} pulls back under $\Psi=\Psi_L\circ\Psi_R\inv$ to the Landau-Ginzburg model $\mathcal F_h$ from Theorem \ref{theo:LGmodel-G/P}, where $h$ and $t$ are related by  $h=t\omega^\vee_m$. 
\end{theorem}

The theorem implies that $\Psi$ maps $\RR$ to $\check X^\circ$. We also expect the following
Claim which we aim to prove in a future version of this paper. 
\begin{cl}
\label{conj:iso}
 $\Psi$ defines an isomorphism from $\RR$ to $\check X^\circ$.
\end{cl}

Let $h=t\omega_m^\vee $ as in the theorem, and define $Z_h:=B_-\dot w_0\cap U_+ e^{h}\dot w_P\dot w_0\inv U_- $.
The super-potential $\mathcal F_h$ pulls back under $Z_h\to G/B_-$ to $\tilde{\mathcal F}_h:Z_h\to \C$ where
\[
 \tilde{\mathcal F}_h(u_1e^{h}\dot w_P \bar u_2)=\sum_{i=1}^m e_i^*(u_1)+\sum_{i=1}^m f_i^*(\bar u_2).
\]
To prove the theorem we need to show  that  $W_t$ pulls back to $\tilde{\mathcal F}_h$ under $Z_h\dot w_0 \overset{\Psi_L}  \longrightarrow P\backslash G=\OG(m+1,2m+1)$. We will do this in two steps.  

We consider two related projective embeddings of $\check X=\OG(m+1,V^*)$, the standard one corresponding to $\bigwedge^{m+1}V^*=V^*_{2\omega_m}$, and the minimal one corresponding to the (right) representation $V^*_\Spin=V^*_{\omega_m}$ of $G=\Spin(V)$ composed with its Veronese embedding. So  
\begin{eqnarray*}
 \pi_1:P\backslash G&\hookrightarrow &\P(\bigwedge{}^{m+1}\ V^*),\\
 Pg &\mapsto & \left <v_{m+1}^*\wedge v_{m+2}^*\wedge\cdots\wedge v_{2m+1}^*\cdot g\right>,\\
 \pi_2:P\backslash G&\hookrightarrow &\P(\Sym^2(V_\Spin^*)), \\
 Pg &\mapsto & \left <(w_\emptyset^*\cdot w_\emptyset^*) \cdot g\right>.
\end{eqnarray*}
The interesting numerators and denominators in $W_t$ are made up of sections in $\Gamma[\mathcal O_{\P(\Sym^2 (V_\Spin^*))}(1)]=\Sym^2 (V_\Spin)$. However the pullback of $\tilde{\mathcal F}_h$ to $\check X$ is not easy to reformulate directly in those terms. It can be  more easily expressed in terms of sections in $ \Gamma[\mathcal O_{\P(\bigwedge^{m+1}V^*)}(1)]=\bigwedge^{m+1}V$, which correspond to $(m+1)\times  (m+1)$-minors. The two embeddings are however related by an embedding of projective spaces coming from the inclusion of representations
\[
 \bigwedge{}^{m+1}\ V^*\hookrightarrow \Sym^2(V_\Spin^*),
\]
Therefore dually we have a surjection of representations
\begin{equation}\label{e:symtowedge}
 \Sym^2 (V_\Spin) \twoheadrightarrow \bigwedge{}^{m+1}V,
\end{equation}
which is the restriction map  $\Gamma[\mathcal O_{\P(\Sym^2(V_\Spin^*))}(1)] \to \Gamma[\mathcal O_{\P(\bigwedge^{m+1}V^*)}(1)]$.

The first step of the proof of the theorem is to express $\tilde{\mathcal F}_h$  in terms of $(m+1)\times  (m+1)$-minors. The second step involves the explicit construction of the above restriction map, and showing that the degree $2$ numerators and denominators in our formula for $W_t$ go to the minors appearing in the first step. From this we go on to deduce that $\psi_L^*W_t$ agrees with $\tilde{\mathcal F}_{h}$.  

\subsection{A formula for $\tilde{\mathcal F}_h$ in terms of minors}\label{s:minors}

\begin{defn} 
 If $g\in \Spin(V)$ we consider it as  acting from the right on $\bigwedge^n V^*$ and from the left on $\bigwedge^n V$ for any $n=1,\dotsc, 2m+1$. The bases $\{v_i^*\}$ and $\{v_i\}$ give rise to bases of $\bigwedge^n V^*$ and $\bigwedge^n V$, and we use the following notation for the matrix coefficients (minors of $g$ acting in the representation $V$). Let $I=\{i_1<\dotsc< i_r\}$ be a set indexing rows, and $J=\{j_1<\dots < j_r\}$ a set indexing columns, then
 \[
  \Delta^{I}_J(g):=\langle v^*_{i_1}\wedge \cdots \wedge v_{i_r}^* \cdot g\ ,\ v_{j_1}\wedge\cdots \wedge v_{j_r} \rangle.
 \]
\end{defn}

We begin by arguing that $\bar u_2$ appearing in  $u_1 e^h \dot{w_P}  \bar u_2 \in Z_h$ can be assumed to lie in $U_-\cap B_+(\dot w^P)\inv B_+ $. This is because we have two birational maps
\begin{eqnarray*}
 \Psi_1: U_-\cap B_+(\dot w^P)\inv B_+ \to P\backslash G: & \bar u_2\mapsto P\bar u_2, &\\
 \Psi_2: B_-\cap U^+e^h\dot w^P U_- \to P\backslash G: & b_-=u_1e^h\dot w_P\bar u_2\mapsto P b_-, &
\end{eqnarray*}
which compose to give $\Psi_1\inv\circ \Psi_2:b_-\mapsto \bar u_2$. This gives a birational map 
\[
 \Psi_1\inv\circ \Psi_2: Z_h\to  U_-\cap B_+(\dot w^P)\inv B_+ .
\] 
Now a generic element $\bar u_2$ in $U_-\cap B_+(\dot w^P)\inv B_+$ can be assumed to have a particular factorisation. 
Let $N:=\binom{m+1}{2}$. The smallest representative $w^P$ in $W$ of $[w_0] \in W/W_P$ has the following reduced expression :
\[
 w^P = (s_m)(s_{m-1} s_m) \dots (s_1 s_2 \dots s_m)=s_{i_1}\dots s_{i_{N}},
\]
It follows that as a generic element of $U_-\cap B_+(\dot w^P)\inv B_+$, the element $\bar u_2$ can be assumed to be written as:
\[
 \big (y_{m}(a_{m,m})    y_{m-1}(a_{m-1,m}) \dots y_{1}(a_{1,m})\big ) \dotsc 
 \big(y_{m}(a_{m,2}) y_{m-1}(a_{m-1,2})\big)y_m(a_{m,1}).
\]
where $a_{i,j}\ne 0$, or equivalently as
\begin{equation}\label{e:u2barfactb}
 \bar u_2=y_{m}(b_{N})    \dots y_{2}(b_{N-m+2}) y_{1}(b_{N-m+1})\dotsc 
 y_{m}(b_3)y_{m-1}(b_2)y_m(b_1).
\end{equation}
with nonzero $b_i$. Note that the $k$-th factor here is $y_{i_{N-k+1}}(b_{N-k+1})$.

We may think of the Pl\"ucker coordinate $\p_\lambda$ as a function on $G$. Then we have the following standard expression for the $\p_{\lambda}$ on factorized elements. 
%
\begin{lemma} \label{l:bislemma} 
 Fix $\lambda$  a strict partition in an $m\times  m$ square, and $w\in W^P$ the corresponding Weyl group element. Note that the length $\ell(w)$ equals $|\lambda|$.  
 Then if $\bar u_2$ is of the form \eqref{e:u2barfactb} we have
 \[
  \p_{\lambda}(\bar u_2)=\sum_{J} b_{{j_1}}\dotsc b_{{j_m}}.
 \]
 where the sum is over subsets $J=\{j_1<j_2<\dotsc < j_m\}$ of $\{1,\dotsc, N\}$  for which $s_{i_{j_1}}\dotsc s_{i_{j_m}}$ is a reduced expression of $w$.
\end{lemma}

\begin{proof} 
 Recall that by definition $\p_{\lambda}(\bar u_2)=\left<w_\emptyset^* \cdot  \bar u_2, w_{\lambda}\right>=w_\emptyset^* ( \bar u_2\cdot  w_{\lambda})$ and $w_\lambda=e_{i_{j_1}}\dotsc e_{i_{j_m}}\cdot w_\emptyset$ if $w=s_{i_{j_1}}\dotsc s_{i_{j_m}}$ is a reduced expression. So in an expansion for $\bar u_2$ the coefficients of  $f_{i_{j_m}}\dotsc f_{i_{j_1}}$ will contribute a summand of $b_{j_1}\dotsc b_{j_m}$ to $\p_{\lambda}(\bar u_2)$.
\end{proof}

\begin{prop}\label{p:eis}
 If $u_1$ and $\bar u_2$ are as above then we have the following identities  
 \begin{equation}\label{eq:ei1first}
  f_m^*(\bar u_2) =\frac{\p_{\rho_{0,+}}(\bar u_2)}{\p_{\rho_0}(\bar u_2)}, 
 \end{equation}
 \begin{equation}
  \label{eq:ei2first}
  e_i^* (u_1) = 0 \text{ for all } 1 \leq i \leq m-1,
 \end{equation}
 \begin{equation}
  \label{eq:ei3first}
  e_m^* (u_1) = e^t \frac{\p_{\rho_{m-1}}(\bar u_2)}{\p_{\rho_m}(\bar u_2)},
 \end{equation}
 where $\rho_0={\emptyset}$ and $\rho_{0,+}=\ydiagram{1}$.
\end{prop}

\begin{proof}
 For \eqref{eq:ei1first} notice that in fact  ${\p_{\emptyset}(\bar u_2)} =1$ and 
 \[
  \p_{{\ydiagram{1}}}(\bar u_2)=\left<w_\emptyset^* \cdot  \bar u_2, w_{\ydiagram{1}}\right>=
  w_\emptyset^* ( \bar u_2 \cdot w_{\ydiagram{1}}).
 \]
 Then \eqref{eq:ei1first} is apparent since $f_m\cdot w_{\ydiagram{1}}= w_{\emptyset} $. In fact \eqref{eq:ei1first} does not depend on the special form of $u_1$ and $\bar u_2$. The equations~\eqref{eq:ei2first} and \eqref{eq:ei3first} are consequences of the Lemmas~\ref{l:ei} and \ref{l:em}, respectively, as well as the Lemma~\ref{l:bislemma}.
\end{proof}

\begin{prop}
 \begin{equation}
  \label{eq:fj}
  f_j^* (\bar u_2) = \frac{\Delta_{j,j+2,\dotsc,j+ m+1}^{m+1,\dotsc, 2m+1}(\bar u_2)}{\Delta_{j+1,\dotsc, j+m+1}^{m+1,\dotsc, 2m+1}(\bar u_2)} 
  \   \text{ for all } 1 \leq j \leq m-1
 \end{equation}
\end{prop}

\begin{proof}
 The result is a consequence of the vanishing of the following minor of $\bar u_2$ :
 \[
  \Delta_{j,j+1,\dotsc,j+ m+1}^{j+1,m+1,\dotsc, 2m+1}(\bar u_2),
 \]
 which is equal to
 \[
  \langle v^*_{j+1}\wedge v^*_{m+1} \cdots \wedge v_{2m+1}^* \cdot g\ ,\ v_{j}\wedge v_{j+1} \cdots \wedge v_{j+m+1} \rangle.
 \]
 Define an element in the enveloping algebra 
 \[
  \underline{e} := \left( e_m^{(a_{1,m})} e_{m-1}^{(a_{1,m-1})} \dots e_1^{(a_{1,1})} \right) \dots \left( e_m^{(a_{m-1,m})} e_{m-1}^{(a_{m-1,m-1})} \right) e_m^{(a_{m,m})},
 \]
 where $a_{i,j} \in \{ 0,1,2 \}$ if $j=m$ and $a_{i,j} \in \{ 0,1 \}$ otherwise. Here $e_i^{(a)}=\frac{1}{a!} e_i^a$. Due to the shape of $\bar u_2$, the minor is zero if for any such $\underline{e}$, $v^*_{j+1}\wedge v^*_{m+1} \cdots \wedge v_{2m+1}^* \cdot \underline{e}$ has zero $v_{j}^*\wedge v_{j+1}^* \cdots \wedge v_{j+m+1}^*$-component. Assume by contradiction that there exists an $\overline{e}$ such that this component is nonzero.
 
 First suppose $j=m-1$. Then since $v^*_{m}\wedge v^*_{m+1} \cdots \wedge v_{2m+1}^* \cdot e_m=0$, the exponent $a_{1,m}$ in $\underline{e}$ has to be zero. Now the $v_{2m+1}^*$ has to be moved to $v_{2m}^*$, which means that $v_{m}^*$ needs to be moved \emph{before} to $v_{m-1}^*$ by an $e_{m-1}$. Since only one $e_1$ appears in the expression of $\underline{e}$, it means that $a_{1,m-1}=1$. Hence $v^*_{m}\wedge v^*_{m+1} \cdots \wedge v_{2m+1}^* \cdot \underline{e}$ is equal to
 \begin{align*}
  v_{m-1}^* \wedge v_{m+1}^* \wedge \dots \wedge v_{2m+1}^* \cdot \left( e_{m-2}^{a_{1,m-2}} \dots e_1^{a_{1,1}} \right) \dots \left( e_m^{a_{m-1,m}} e_{m-1}^{a_{m-1,m-1}} \right) e_m^{a_{m,m}}.
 \end{align*}
 Since $v_{m-1}^* \wedge v_{m+1}^* \wedge \dots \wedge v_{2m+1}^* \cdot e_i = 0$ for all $1 \leq i \leq m-2$, it follows that $a_{1,m-2}=\dots=a_{1,2}=a_{1,1}=0$, which means that the $v_{2m+1}^*$ can never be moved to $v_{2m}^*$. Hence there exists no $\underline{e}$ such that $v^*_{j+1}\wedge v^*_{m+1} \cdots \wedge v_{2m+1}^* \cdot \underline{e}$ has nonzero $v_{j}^*\wedge v_{j+1}^* \cdots \wedge v_{j+m+1}^*$-component.
 
 Now suppose $j<m-1$. $v_{2m+1}^*$ has to be moved to $v_{2m}^*$ by the only $e_1$ in the expression of $\underline{e}$, hence $a_{1,1}=1$. But $v_{m+1}^*, \dots, v_{2m}^*$ need to be moved before, hence $a_{1,i}=1$ for $1 \leq i \leq m-1$ and $a_{1,m}=2$. It follows that $v^*_{j+1}\wedge v^*_{m+1} \cdots \wedge v_{2m+1}^* \cdot \underline{e}$ is equal to
 \begin{align*}
  \left(v^*_{j+1}\wedge v^*_{m} \cdots \wedge v_{2m}^* + v^*_{1}\wedge v^*_{m} \cdots \wedge v_{m-j}^* \wedge v_{m+2-j}^* \wedge \dots \wedge v_{2m+1}^* \right) \cdot \underline{e}',
 \end{align*}
 where 
 \[
  \underline{e}' := \left( e_m^{a_{2,m}} e_{m-1}^{a_{2,m-1}} \dots e_2^{a_{2,2}} \right) \dots \left( e_m^{a_{m-1,m}} e_{m-1}^{a_{m-1,m-1}} \right) e_m^{a_{m,m}}.
 \]
 Then
 \begin{align*}
  v^*_{1}\wedge v^*_{m} \cdots \wedge v_{m-j}^* \wedge v_{m+2-j}^* \wedge \dots \wedge v_{2m+1}^* \cdot \underline{e}'
 \end{align*}
 has clearly no non-zero $v_{j}^*\wedge v_{j+1}^* \cdots \wedge v_{j+m+1}^*$-component, hence we focus on $v^*_{j+1}\wedge v^*_{m} \cdots \wedge v_{2m}^* \cdot \underline{e}'$.
 
 If $j=m-2$, then $v_{2m}^*$ has to be moved to $v_{2m-1}^*$ by the only $e_2$ in $\underline{e}'$. Hence $a_{2,2}=1$. But $v^*_{m-1}\wedge v^*_{m} \cdots \wedge v_{2m}^* \cdot e_m=0$, which means that $a_{2,m}=0$. It follows that $v_{m+1}^*$ cannot be moved to $v_m^*$ before having to move the $v_{2m}^*$, and hence that a suitable $\underline{e}$ does not exist.
 
 Finally if $j \leq m-3$, then $v^*_{j+1}\wedge v^*_{m} \cdots \wedge v_{2m}^* \cdot e_i=0$ for all $j+1 \leq i \leq m$, hence $a_{2,j+1}=\dots=a_{2,m}=0$. It follows that the $v_{m+1-j}^*$ cannot be moved before the $v_{2m}^*$ has to be by the only remaining $e_2$ in $\underline{e}'$. This concludes the proof of the minor vanishing.
 
 To prove the proposition, we only need to expand this vanishing minor with respect to the $(j+1)$-st row. Indeed, due to $\bar u_2$ being lower triangular, this row has only two non-zero entries :  $1$ on the $(j+1)$-st column and $f_j^*(\bar u_2)$ on the $j$-th column.
\end{proof}



\subsection{The Clifford Algebra and  homogeneous coordinates}

\subsubsection{Setting}

In this section we study the surjection of representations from \eqref{e:symtowedge}, that is
\[
 \pi: \Sym^2 (V_\Spin)\to \bigwedge{}^{m+1}V,
\]
which is also interpreted as the restriction map of homogeneous coordinates 
\[
 \Gamma[\mathcal O_{\P(\Sym^2(V_\Spin^*))}(1)] \to \Gamma[\mathcal O_{\P(\bigwedge^{m+1}V^*)}(1)].
\]

Of course in representation-theoretic terms  the map $\pi$ exists just because $\bigwedge{}^{m+1}V$ is irreducible with highest weight $2\omega_m$ and this highest weight also occurs in  $\Sym^2 (V_\Spin)$ with multiplicity one. But in order to compute with this map we will need to use a more intrinsic construction. We first note the following auxiliary lemma, whose proof is straightforward.
\begin{lemma}
 \label{l:dual}
  The isomorphism 
  \[
   \begin{array}{ccc}
    \delta: \Vspin & \rightarrow & \Vspin^* \\
    v_\lambda & \mapsto & (-1)^{|\lambda|} v_{\PD(\lambda)}
   \end{array}
  \]
  is $\mathfrak{so}(V)$-equivariant.
\end{lemma}

For the construction of the map $\pi$ first we define an equivariant embedding
\[
 \iota_{V_\Spin}:\Sym^2(V_\Spin)\hookrightarrow V_\Spin \otimes V_\Spin\overset{ \delta\otimes\id_{V_\Spin} } \longrightarrow V_\Spin^* \otimes V_\Spin = \End(V_\Spin). 
\]
Then there are two subtly different cases to distinguish.

\paragraph{Case 1:} If $m$ is odd then we construct $\pi$ as follows. 
 Applying the constructions from Section~\ref{s:notations} 
 we have an isomorphism  of representations \eqref{e:OddToEnd},
 \[
  \kappa_-\inv: \End(V_\Spin)\to\Cl^-(V)\to\bigoplus_{k=0}^{m}\bigwedge{}^{2k+1}\ V.
 \] 
 Because $m$ is odd we have a projection onto the summand with $k=\frac{m-1}2$,
 \[
  \pr_{\wedge^{m}}: \bigoplus_{k=0}^{m}\bigwedge{}^{2k+1}\ V\to \bigwedge{}^{m}\ V.
 \]
 By contracting with $(-1)^{\frac{m(m+1)}{2}} v_1^*\wedge\cdots \wedge v_{2m+1}^*$ we get an equivariant isomorphism
 \[
  c:~\bigwedge{}^{m}\ V\to\bigwedge{}^{m+1}\ V^*.
 \]
 Finally, we have an equivariant isomorphism
 \[
  d:~\bigwedge{}^{m+1}\ V^* \to \bigwedge{}^{m+1}\ V
 \]
 defined using the isomorphism $V \cong V^*$ given by the quadratic form and made explicit in the end of Subsection \ref{s:notations}. Composing $\iota_{V_\Spin}$ with these four maps gives us our homomorphism of representations 
 \[
  \pi: \Sym^2(V_\Spin)\longrightarrow  \bigwedge{}^{m+1}\ V.
 \]

\paragraph{Case 2:} Suppose $m$ is even. In this case we use the even part of the Clifford algebra of $V$, namely
 we use the inverse of the isomorphism from \eqref{e:EvenToEnd}
 \[
  \kappa_+\inv: \End(V_\Spin)\to\Cl^+(V)\to\bigoplus_{k=0}^{m}\bigwedge{}^{2k}\ V.
 \] 
 Since $m$ is even we have a projection onto the middle summand, $k=\frac{m}2$,
 \[
  \pr_{\wedge^{m}}: \bigoplus_{k=0}^{m}\bigwedge{}^{2k}\ V^*\to \bigwedge{}^{m}\ V.
 \]
 Finally we use the isomorphism of representations $c$ as in Case~1,
 \[
  c: \bigwedge{}^{m}\ V\overset\sim\longrightarrow \bigwedge{}^{m+1}\ V^*
 \]
 as well as the map
 \[
  d:~\bigwedge{}^{m+1}\ V^* \to \bigwedge{}^{m+1}\ V.
 \]

 Composing $\iota_{V_{\Spin}}$ with these four maps gives us our homomorphism of representations 
 \[
  \pi:\Sym^2(V_\Spin)\longrightarrow  \bigwedge{}^{m+1}\ V
 \]
 in the case where $m$ is even.

\subsubsection{Statement}

\begin{defn}
 Corresponding to the quadratic denominators in $W_t$ we define elements of $\Sym^{2}(\Vspin)$ by
 \[
  \D_{(j)} := \sum_{I } (-1)^{s(I)} w_{\rho_{m+1-j}^I} w_{\mu_{m+1-j}^I}
 \]
 and
 \[
  \N_{(j)} := \sum_{I } (-1)^{s(I)} w_{\rho_{m+1-j,+}^I} w_{\mu_{m+1-j,+}^I}
 \]
 where the sums are over all subsets $I\subset \{1,\dots,m+1-j\}$ and $j=2,\dotsc, m$.
\end{defn}

We will prove that for all $j=2,\dotsc, m$ :
\begin{prop}
\label{p:SymToMinor}
 \[
  \sum_{I } (-1)^{s(I)} \p_{\rho_{m+1-j}^I}(\bar u_2) \p_{\mu_{m+1-j}^I}(\bar u_2) = \Delta_{m+2-j,\dotsc, 2m+2-j}^{m+1,\dotsc, 2m+1}(\bar u_2)
 \]
 and
 \[
  \sum_{I } (-1)^{s(I)} \p_{\rho_{m+1-j,+}^I}(\bar u_2) \p_{\mu_{m+1-j,+}^I}(\bar u_2) = \Delta_{m+1-j,m+3-j\dotsc, 2m+2-j}^{m+1,\dotsc, 2m+1}(\bar u_2)
 \]
 where the sums are over all subsets $I\subset \{1,\dots,m+1-j\}$.
\end{prop}

\begin{remark} Note that this proposition gives us an alternative definition of $\check X^\circ$ in terms of  non-vanishing of minors.   
\end{remark}
\subsubsection{Proof}

To prove Proposition \ref{p:SymToMinor}, we will need to compare $\D_{(j)},\N_{(j)} \in\Sym^{2}(\Vspin)$ to the elements of $\bigwedge^{m+1}V$ defined below.
\begin{defn}\label{d:jwedge} 
 Inside the exterior power $\bigwedge^{m+1}V$, if $2 \leq j \leq m$ we consider the elements
 \begin{align*}
  v_{(j)}^{\wedge} &:= v_j \wedge \dots \wedge v_{j+m} \\
  v_{(j),+}^{\wedge} &:= v_{j-1} \wedge v_{j+1} \wedge \dots \wedge v_{j+m}
 \end{align*}
 of $\bigwedge^{m+1} V$.
\end{defn}

We will show :
\begin{prop}\label{p:DenomProj}
 The projection map $\pi:\Sym^2(V_\Spin)\longrightarrow  \bigwedge^{m+1}\ V$ takes $\D_{(j)}$ to $v^\wedge_{(j)}$ and $\N_{(j)}$ to $v_{(j),+}^{\wedge}$.
\end{prop}
We will in fact prove this proposition only for the denominators $\D_{(j)}$, the case of the numerators $\N_{(j)}$ being extremely similar. 

\begin{defn} 
 If $I= \{ 1 \leq i_1 < \dots < i_r \leq 2m+1 \}$, we define $v_I$ to be the product $v_{i_1} \cdot \dots \cdot v_{i_r}$ in $\Cl(V)$. For $I=\{j,j+1,\dotsc, j+m\}$ we also denote $v_I$ by $v_{(j)}$, so $v_{(j)}=v_j v_{j+1}\cdots  v_{j+m}$. Moreover, if $L$ is a subset of $\{ j, \dots, m \}$, we write $v_{(j)}^L$ for the Clifford algebra element obtained from the product $v_{(j)}$
 by removing all of the factors $v_l$ and $\bar v_l=v_{2m+2-l}$ for which $l\in L$.
\end{defn}

\begin{lemma} 
 The map $\iota_{V_\Spin}: \Sym^2(V_\Spin)\hookrightarrow \End(V_\Spin)$ maps $\D_{(j)}$ to
 \begin{equation}
  \label{eq:denom2end}
  \beta_{m,j} \cdot \sum_{I } \left[ w_{\mu_{j-1}^I}^* \otimes w_{\mu_{m+1-j}^I} + (-1)^{(m+1-j)(j-1)} w_{\rho_{j-1}^I}^* \otimes w_{\rho_{m+1-j}^I} \right]
 \end{equation}
 where 
 \[
  \beta_{m,j}:=\frac{(-1)^{\frac{(m+1-j)(m+2-j)}{2}}}{2}
 \]
 and the sum is over all subsets $I$ of $\{1,\dots,m+1-j\}$.
\end{lemma}

\begin{proof}
 First $w_{\rho_{m+1-j}^I} w_{\mu_{m+1-j}^I}$ maps to 
 \[
  \frac{1}{2}(w_{\rho_{m+1-j}^I} \otimes w_{\mu_{m+1-j}^I} + w_{\mu_{m+1-j}^I} \otimes w_{\rho_{m+1-j}^I}) \in V_\Spin \otimes V_\Spin.
 \]
 Then according to Lemma \ref{l:dual} :
 \begin{align*}
  w_{\rho_{m+1-j}^I} &\mapsto (-1)^{\frac{(m+1-j)(m+2-j)}{2}-s(I)} w_{\mu_{j-1}^I}^* \in V_\Spin^* \\
  w_{\mu_{m+1-j}^I} &\mapsto (-1)^{\frac{m(m+1)}{2}-\frac{j(j-1)}{2}+s(I)} w_{\rho_{j-1}^I}^* \in V_\Spin^*,
 \end{align*}
 hence the result.
\end{proof}

We now need to map the element \eqref{eq:denom2end} to the Clifford algebra of $V$. 

\begin{proposition} 
\label{p:DinCl}
\begin{multline}
\D_{(j)} \mapsto 
\frac{(-1)^{\frac{m(m+1)}{2}}}{2} \left[ 2 v_{1,\dots,m+1-j,2m+3-j,\dots,2m+1} \right. + \\ \left . \sum_{I \subsetneq \{ 1,\dots,m+1-j \}} \left( \prod_{l\in \{ 1,\dots,m+1-j \} \setminus I} (-1)^l \right) v_{I\cup\{ 2m+3-j,\dots,m+j\} \cup \overline{I}}\right] \in \Cl(V).
\end{multline}
\end{proposition}

\begin{proof}
 We assume $j > \frac{m+1}{2}$, the other case being symmetric. For convenience, let us denote the right-hand side as $A_{(j)} \in \Cl(V)$. Because of the definition of the Clifford algebra :
 \[
  v_{1,\dots,m+1-j,2m+3-j,\dots,2m+1} = (-1)^{m(m+1-j)} v_{\{1,\dots,m+1-j\} \cup \overline{\{ 1,\dots,m+1-j\}}}\, t_{(j)},
 \]
 where $t_{(j)} = v_{2m+3-j} \dots v_{m+j}$. Similarly 
 \[
  v_{I\cup\{ 2m+3-j,\dots,m+j\} \cup \overline{I}} = (-1)^{m|I|} v_{I \cup \overline{I}}\, t_{(j)}.
 \]
 We will use two lemmas :
 \begin{lemma}
 \label{l:I-Ibar}
  Let $I$ be a subset of $\left\{ 1, \dots, m \right\}$. Then
  \[
   v_{I\cup \overline{I}} \mapsto \left( \prod_{i \in I} \epsilon(i) \right) \sum_{L} w_L^* \otimes w_L \in \End(\Vspin),
  \]
  where the sum is over all subsets $L$ of $\left\{ 1, \dots, m \right\}$ containing $I$.
 \end{lemma}
 
 \begin{proof}[Proof of lemma \ref{l:I-Ibar}]
  First notice that 
  \begin{align*}
   v_{\overline{i}} \cdot w_L = \begin{cases}
                                 0 & \text{if $i\not\in L$} \\
                                 (-1)^{\#\{ l \in L \mid l < i \}} \epsilon(i) w_{L \setminus \{ i \}} & \text{otherwise,}
                                \end{cases}
  \end{align*}
  and
  \begin{align*}
   v_i v_{\overline{i}} \cdot w_L = \begin{cases}
                                     0 & \text{if $i\not\in L$} \\
                                     \epsilon(i) w_{L} & \text{otherwise}.
                                    \end{cases}
  \end{align*}
  Hence $v_{I\cup \overline{I}}$ is zero unless $L \supset I$. Now assume $L\supset I$ and write $I=\{ i_1 < i_2 < \dots < i_r \}$. From the definition of the Clifford algebra, it follows that $v_{I\cup \overline{I}}=\prod_{p=1}^r v_{i_p} v_{\overline{i_p}}$. Hence :
  \[
   v_{I\cup \overline{I}} \cdot w_L = \left(\prod_{p=1}^r \epsilon(i_p)\right) w_L. 
  \]
  The claim follows.
 \end{proof}

 \begin{lemma}
 \label{l:tj}
  The element $t_{(j)} = v_{2m+3-j} \dots v_{m+j}$ of $\Cl(V)$ maps to
  \[
   \left( \prod_{p=m+2-j}^{j-1} \epsilon(p) \right) \sum_{K_1,K_2} (-1)^{m|K_1|} w_{K_1 \cup \{ m+2-j, \dots, j-1 \} \cup K_2}^* \otimes w_{K_1 \cup K_2} \in \End(\Vspin),
  \]
  where $K_1$ is any subset of $\{1,\dots,m+1-j\}$ and $K_2$ is any subset of $\{ j,\dots,m \}$.
 \end{lemma}
 
 \begin{proof}[Proof of lemma \ref{l:tj}]
  As in the proof of Lemma \ref{l:I-Ibar}, we notice that $t_{(j)} \cdot w_L=0$ if $L \not \supset \{ m+2-j, \dots, j-1 \}$. Now write $L=L_1\cup \{ m+2-j, \dots, j-1 \} \cup L_2$, where $L_1 \subset \{1,\dots,m+1-j\}$ and $L_2 \subset \{ j,\dots,m \}$. We have 
  \[
   v_{m+j} \cdot w_L = (-1)^{m|L_1|} \epsilon(m+2-j) w_{L_1 \cup \{ m+3-j, \dots, j-1 \} \cup L_2}.
  \]
  Recursively, we obtain :
  \[
   t_{(j)} \cdot w_L = (-1)^{m|L_1|} \left( \prod_{p=m+2-j}^{j-1} \epsilon(p) \right) w_{L_1 \cup L_2},
  \]
  hence the lemma.
 \end{proof}
 
 Now to prove Proposition \ref{p:DinCl}, first assume $L=\{ 1,\dots, j-1 \} \cup L_2$, where $L_2 \subset \{ j, \dots, m\}$. Then
 \[
  v_{1,\dots,m+1-j,2m+3-j,\dots,2m+1} \cdot w_L = \left( \prod_{p=1}^{j-1} \epsilon(p) \right) w_{L_1 \cup L_2},
 \]
 and
 \[
  \left( \prod_{l \in \{ 1,\dots,m+1-j \} \setminus I} (-1)^l \right) v_{I\cup\{ 2m+3-j,\dots,m+j\} \cup \overline{I}} \cdot w_L 
 \]
 is equal to
 \[
  \left( \prod_{p=1}^{j-1} \epsilon(p) \right) (-1)^{|I|}(-1)^{m+1-j} w_{L_1 \cup L_2}.
 \]
 Hence 
 \begin{align*}
  A_{(j)} \cdot w_L &= \left( \prod_{p=1}^{j-1} \epsilon(p) \right) \left[ 2 + (-1)^{m+1-j}\sum_{I \subsetneq \{ 1,\dots,m+1-j\} } (-1)^{|I|} \right] w_{L_1 \cup L_2} \\
		    &= \left( \prod_{p=1}^{j-1} \epsilon(p) \right) w_{L_1 \cup L_2}.
 \end{align*}
 Now assume $L=L_1 \cup \{ m+2-j,\dots, j-1 \} \cup L_2$, where $L_1 \subsetneq \{1,\dots,m+1-j \}$ and $L_2 \subset \{ j, \dots, m\}$. Then
 \[
  A_{(j)} \cdot w_L =  \left( \prod_{p=1}^{j-1} \epsilon(p) \right) (-1)^{m|L_1|} (-1)^{(m+1)(m+1-j)} \sum_{I \subset L_1} (-1)^{|I|}  w_{L_1 \cup L_2}.
 \]
 Finally :
 \begin{equation*}
  A_{(j)} \cdot w_L = \begin{cases}
                       0 & \text{if $L_1 \neq \emptyset$} \\
                       \left( \prod_{p=1}^{j-1} \epsilon(p) \right) (-1)^{(m+1)(m+1-j)} w_{L_2} & \text{otherwise.}
                      \end{cases}
 \end{equation*}
 Looking precisely at the expression of $\D_{(j)}$ in $\End(\Vspin)$, this concludes the proof of the proposition.
\end{proof}

\begin{corollary}
\label{coro:SymToWedge}
 We have :
 \[
  \pr_{\wedge^{m}} \circ \kappa_{\pm}^{-1} \circ \iota_{\Vspin} (\D_{(j)}) = (-1)^{\frac{m(m+1)}{2}} v_1 \wedge \dots \wedge v_{m+1-j} \wedge v_{2m+3-j} \wedge \dots \wedge v_{2m+1}
 \]
 where $\kappa_{\pm}$ is $\kappa_-$ if $m$ is odd and $\kappa_+$ otherwise.
\end{corollary}
 
\begin{proof}
 The result is a simple consequence of Proposition \ref{p:DinCl} and of the definition of the antisymmetrisation maps \eqref{e:evenPartsIso} and \eqref{e:oddPartsIso}.
\end{proof}
 
We can now prove Proposition \ref{p:DenomProj} :
\begin{proof}[Proof of Proposition \ref{p:DenomProj}]
 From Corollary \ref{coro:SymToWedge}, we know that $\D_{(j)}$ maps to $(-1)^{ \frac{m(m+1)}{2} } v_1 \wedge \dots \wedge v_{m+1-j} \wedge v_{2m+3-j} \wedge \dots \wedge v_{2m+1}$ in $\bigwedge^m V$. Now the latter element is mapped by the contraction $c$ to
 \[
  (-1)^{(m+1)(j-1)} v_{m+2-j}^* \wedge \dots \wedge v_{2m+2-j}^*.
 \]
 Then we map this to $\bigwedge^{m+1} V$ using the isomorphism $d$. We have
 \begin{align*}
  v_{m+2-j}^* \wedge \dots \wedge v_{2m+2-j}^* &\mapsto \left( \prod_{i=m+2-j}^m \epsilon(i) \right) \left( \prod_{k=j}^m \epsilon(k) \right) v_{j+m} \wedge v_{j+m+1} \wedge \dots \wedge v_j \\
   &\mapsto (-1)^{j^2+m^2-m j+1} v_{(j)}^\wedge.
 \end{align*}
 Now
 \[
  \D_{(j)} \mapsto v_{(j)}^\wedge,
 \]
 which concludes the proof.
\end{proof}

\section{Some relations in the quantum cohomology}
\label{s:relations}

In \cite{rietsch}, the second author proved an isomorphism between the quantum cohomology ring of $X=G^\vee/P^\vee$ and the Jacobi ring of the LG-model $(\mathcal{R},\mathcal F_h)$ (either at fixed quantum parameter $q=e^h$ as in Corollary~\ref{theo:LGmodel-G/P} or over the ring $\C[q,q\inv]$). By Theorem \ref{theo:W} together with Claim \ref{conj:iso} our LG-model $(\check{X},W_t)$ should be isomorphic to this one, and therefore related to the quantum cohomology ring of $LG(m)$ in the same way.
Therefore we expect the denominators appearing in the expression of $W_t$, once written with Schubert classes replacing the Pl\"ucker coordinates, to represent invertible elements in this quantum cohomology ring. 
We have a precise conjecture for which elements these are.

\begin{conjecture}
 The following relation holds in the quantum cohomology of $\LG(m)$ for all $1 \leq l \leq m-1$ :
 \begin{equation}
 \label{e:qrelation}
  \underset{J\subset\left\{1,\dots,l\right\}}{\sum}(-1)^{s(J)}\sigma_{\rho_l^{J}} \star \sigma_{\mu_l^J} = q^l.
 \end{equation}
\end{conjecture}

\begin{remark}
 If $l=1$, the relation \eqref{e:qrelation} is a consequence of the quantum Chevalley formula \ref{th:quantChevalley}. Indeed, this formula implies that
 \[
  \sigma_1 \star \sigma_m = \sigma_{m,1} + q,
 \]
 which, rewritten as 
 \[
  \sigma_1 \star \sigma_m - \sigma_\emptyset \star \sigma_{m,1} = q,
 \]
 is exactly the relation \eqref{e:qrelation} with $l=1$.
 For $l>1$ however, to the best of the authors' knowledge, the relations \eqref{e:qrelation} are new.
\end{remark}

\section{The B-model connection}
In this section we briefly state an explicit mirror symmetry conjecture for our superpotential $W_t$. Namely the conjecture asserts that a Gauss-Manin connection associated to $W$ should recover connections defined on the $A$-model side by Dubrovin and Givental, see 
\cite{Dub:2DTFT, Givental:EqGW,CoxKatz}.

Let $X=LG(m)$. Consider  $H^*(X,\C[\hbar, e^t])$ as space of sections on a trivial
bundle with fiber $H^*(X)$ and let 
\begin{eqnarray}
{}^A\nabla_{\partial_t} S&:=& \frac{d S}{dt} -
 \frac{1}{\hbar}
 \sigma^{\ydiagram{1}}\star_{e^t}S\\
{}^A \nabla_{\hbar\partial_\hbar} S&:=& \hbar\frac{\partial S}{\partial\hbar} + \frac{1}{\hbar} c_1(TX)\star_{e^t}S
\end{eqnarray}
define a meromorphic flat connection on this bundle.\footnote{ We are using the convenient notation $e^t$ for $q$ and $\partial_t$ for 
$q\partial_q$. Also, for simplicity of the statement of the conjecture, we have removed the grading operator contained in Dubrovin's original definition of ${}^A \nabla_{\hbar\partial_\hbar} $.}  This is our $A$-model side.

For the $B$-model let $N=\frac{m(m+1)}2$ denote the dimension of $\check X$.  Recall that $\check X^\circ$ is $\OG(m+1,2m+1)$ with an anti-canonical divisor removed. Therefore there is an up to scalar unique non-vanishing holomorphic N-form on $\check X^\circ$ which we will fix and call $\omega$. Let $\Omega^k(\check X^\circ)$ denote the space of all holomorphic $k$-forms.
\begin{defn}
Define the $\C[\hbar,e^t]$-module
\[
G_0^{W_t}:=\Omega^N(X)[\hbar,e^t]/( \hbar d +  dW_t\wedge - ) \Omega^{N-1}(X)[\hbar,e^t].
\]
It has a meromorphic (Gauss-Manin) connection given by
\begin{eqnarray}
{}^B\nabla_{\partial_t} [\alpha]&=&\frac\partial{\partial t} [\alpha] - \frac{1}{\hbar}[\frac{\partial W_t}{\partial t}\alpha],\\
{}^B\nabla_{\partial_\hbar} [\alpha]&=&\frac\partial{\partial \hbar} [\alpha] +\frac{1}{\hbar^2} [ W_t\, \alpha].
\end{eqnarray}
\end{defn}

We conjecture that the function $W_t$ is cohomologically tame \cite{sabbahtame} and the elements 
$[\p_\lambda \omega]$ freely generate $G_0^{W_t}$, where the $\p_{\lambda}$'s are the  Pl\"ucker coordinate on $OG^{co}(m+1,V^*)$ and $\lambda$ runs through the strict partitions inside an $m\times m$ box.

Independently of this we conjecture
the following.

\begin{conjecture}\label{c:main}
The differential operators $\hbar{\, }^B\nabla_{\partial t}$ and $\hbar{\, }^B\nabla_{\hbar\partial{\hbar}}$
preserve the $\C[\hbar, e^t]$-submodule $\bar G_0^{W_t}$ of $G_0^{W_t}$ generated by the $[\, \p_\lambda \omega]$. Moreover the assignment  $\sigma^{\lambda}\mapsto[\p_{\lambda}\omega]$ defines an isomorphism of $H^*(X,\C[\hbar, e^t])$ with $\bar G_0^{W_t}$  under which ${}^A\nabla$ is identified with ${}^B\nabla$. 
\end{conjecture}

\appendix
\section{Appendix}

We may give also a Laurent polynomial expression for $W_t$ restricted to a particular torus. Namely let us pull back $W_t$ to the open subset of $\check X$ defined as the image of the map $(\C^*)^N\hookrightarrow P\backslash G$ which sends $(b_1,\dotsc, b_N)$ to
$P\bar u_2$, where as in Section~\ref{s:minors}
\begin{equation}\label{e:u2barfactbb}
\bar u_2=y_{m}(b_{N})    \dots y_{2}(b_{N-m+2}) y_{1}(b_{N-m+1})\dotsc 
 y_{m}(b_3)y_{m-1}(b_2)y_m(b_1).
 \end{equation}

\begin{prop}[Laurent polynomial restriction of $W_t$]
 \label{prop:W1}
 The Landau-Ginzburg model $W_t$ of $X=\LG(m)$ defined in Theorem \ref{theo:LGmodel-G/P} restricts to the open torus defined above to give
 \[
\tilde W_t(b_1,\dotsc, b_N) = \sum_{j=1}^N b_j + e^t\ \frac{\mathcal N(b_1,\dotsc, b_N)}{\prod_{j=1}^{N}b_j},
 \]
 where 
 \[
 \mathcal N(b_1,\dotsc, b_N) := \sum b_{j_{i_1}} \dots b_{j_{i_{N-m}}},
 \]
 and the sum is over all subsets $\{ i_1 < \dots <i_{N-m}\}$ of $\{1,\dotsc, N\}$ such that 
 $ (s_{j_{i_1}} \dots s_{j_{i_N}})s_1\dotsc s_m $ is a reduced
 expression for $w^P$. 
\end{prop}

\begin{proof}
We will rename the coordinates $b_i$ when convenient by $a_{i,j}$, in terms of which 
$\bar u_2$  is given by
\[
\big (y_{m}(a_{m,m})    y_{m-1}(a_{m-1,m}) \dots y_{1}(a_{1,m})\big ) \dotsc 
 \big(y_{m}(a_{m,2}) y_{m-1}(a_{m-1,2})\big)y_m(a_{m,1}).
 \]

 As a consequence of the shape of $\bar u_2$ and the definition of the $y_i$, we immediately obtain :
 \begin{equation}
  \label{eq:ei1}
   f_i^* (\bar u_2) = \sum_{j=m+1-i}^m a_{j}^{(i)}.
 \end{equation}
 We now need to compute the $e_i^* (u_1)$, where $u_1$ is such that $u_1 e^h \dot{w_P}  \bar u_2 \in B_- \dot{w_0}$.
 \begin{lemma}
  \label{l:ei}
   \begin{equation}
    \label{eq:ei2}
     e_i^* (u_1) = 0 \text{ for all } 1 \leq i \leq m-1
   \end{equation}
  \end{lemma}
  \begin{proof}[Proof of Lemma~\ref{l:ei}]
   From \cite{rietsch}, we know that 
   \begin{align*}
    e_i^* (u_1) & = \frac{\langle u_1^{-1} v_{\omega_i}^- , e_i \cdot v_{\omega_i}^- \rangle}{\langle u_1^{-1} v_{\omega_i}^- , v_{\omega_i}^- \rangle} \\
		     & = \frac{\langle e^h \dot{w_P} \bar u_2 \dot{w_0}^{-1} v_{\omega_i}^- , e_i \cdot v_{\omega_i}^- \rangle}{\langle e^h \dot{w_P} \bar u_2 \dot{w_0}^{-1} v_{\omega_i}^- , v_{-\omega_i}^- \rangle} \\
		     & = \frac{\langle e^h \dot{w_P} \bar{u}_2 v_{\omega_i}^+ , e_i \cdot v_{\omega_i}^- \rangle}{\langle e^h \dot{w_P} \bar{u}_2 v_{\omega_i}^+ , v_{\omega_i}^- \rangle}.
   \end{align*}
   Now $e_i^* (u_1) = 0$ if and only if $\langle \overline{u}_2 v_{\omega_i}^+ , \dot{w}_P^{-1} e_i \cdot v_{\omega_i}^- \rangle = 0$. The vector $\dot{w}_P^{-1} e_i \cdot v_{\omega_i}^-$ is in the $\mu$-weight space of the $i$-th fundamental representation, where $\mu := w_P^{-1} s_i (-\omega_i)$. Moreover, $\overline{u}_2 \in B_+ (\dot w^P)^{-1}  B_+$, hence it can only have non-zero components down to the weight space of weight $(w^P)^{-1} (\omega_i)=w_P^{-1}(-\omega_i)$. However, $\mu$ is lower than $w_P^{-1}(-\omega_i)$ when $i \neq m$. 
  \end{proof}
  We are left with computing $e_m^* (u_1)$ :
  \begin{lemma}
  \label{l:em}
   \begin{equation}
    \label{eq:ei3}
    e_m^* (u_1) = e^t\ \frac{\mathcal N(b_1,\dotsc, b_N)}{\prod_{j=1}^{N}b_j}
   \end{equation}
  \end{lemma}
  \begin{proof}[Proof of Lemma~\ref{l:em}]
   As in the proof of Lemma~\ref{l:ei}, we have
   \begin{align*}
    e_m^* (u_1) & = \frac{\langle e^h\dot{w_P} \bar{u}_2 v_{\omega_m}^+ , e_m \cdot v_{\omega_m}^- \rangle}{\langle e^h\dot{w_P} \bar{u}_2 v_{\omega_m}^+ , v_{\omega_m}^- \rangle} \\
		     & = (\omega_m + \alpha_m - \omega_m)(e^h) \frac{\langle \dot{w_P} \bar{u}_2 v_{\omega_m}^+ , e_m \cdot v_{\omega_m}^- \rangle}{\langle \dot{w_P} \bar{u}_2 v_{\omega_m}^+ , v_{\omega_m}^- \rangle} \\
		     & = 
		    e^t \frac{\langle \dot{w_P} \bar{u}_2 v_{\omega_m}^+ , e_m \cdot v_{\omega_m}^- \rangle}{\langle \dot{w_P} \bar{u}_2 v_{\omega_m}^+ , v_{\omega_m}^- \rangle}.
   \end{align*}
   Indeed, $\alpha_m(e^h)=e^t$. 
    Moreover, $\langle \dot{w_P} \overline{u}_2 v_{\omega_m}^+ , v_{\omega_m}^- \rangle = \langle \overline{u}_2 v_{\omega_m}^+ , \dot{w_P}^{-1} v_{\omega_m}^- \rangle = \langle \overline{u}_2 v_{\omega_m}^+ , v_{\omega_m}^- \rangle$. Now the only way to go from the lowest weight vector $v_{\omega_m}^-$ of the $m$-th fundamental representation to the highest $v_{\omega_m}^+$ is to apply $w_0$. Since $\overline{u}_2 \in B (w^P)^{-1} B$, it follows that we need to take all factors of $\overline{u}_2$, hence $\langle \dot{w_P} \overline{u}_2 v_{\omega_m}^+ , v_{\omega_m}^- \rangle = \prod_{j=1}^{N}b_j$.
 
   Now we prove that $\langle \dot{w_P} \overline{u}_2 v_{\omega_m}^+ , e_m \cdot v_{\omega_m}^- \rangle = \mathcal N(b_1,\dotsc, b_N)$. Indeed :
   \[
    \langle \dot{w_P} \overline{u}_2 v_{\omega_m}^+ , e_m \cdot v_{\omega_m}^- \rangle = \langle \overline{u}_2 v_{\omega_m}^+ , \dot{w_P}^{-1} e_m \cdot v_{\omega_m}^- \rangle,
   \]
   and the weight of the vector $\dot{w_P}^{-1} e_m \cdot v_{\omega_m}^-$ is $\mu' := \frac{1}{2}(\epsilon_1 - \epsilon_2 - \dots - \epsilon_m)$. Now consider the Weyl group element
   \[
    w' := s_m (s_{m-1} s_m) \dots (s_2 \dots s_{m-1} s_m).
   \]
We have
   \[
    w' \cdot \omega_m = \frac{1}{2}(\epsilon_1 - \epsilon_2 - \dots - \epsilon_m).
   \]
   Hence the way to the $\mu'$-weight space is through one of the reduced expression for $w'$, which concludes the proof of the claim.
  \end{proof}
Now the proof of Proposition \ref{prop:W1} follows immediately from Theorem \ref{theo:LGmodel-G/P} and the equations \eqref{eq:ei1}, \eqref{eq:ei2} and \eqref{eq:ei3}.
\end{proof}

The expression for the Landau-Ginzburg model in Proposition \ref{prop:W1} is quite close to the usual expression for the Landau-Ginzburg model of projective space $\P^n$, which looks like :
\[
 W^{\P^n}_t = x_1 + x_2 + \dots + x_n + \frac{e^t}{x_1 x_2 \dots x_n}.
\]
Indeed, It is the sum of as many parameters as the dimension of the variety, plus a more complicated $e^t$-term depending on those parameters. To the best of our knowledge, this expression is new for $\LG(m)$ with $m>2$. However, for the three-dimensional quadric $\LG(2)$, we obtain :
\[
 W^{LG(2)}_t = a_{2,1} + a_{1,2} + a_{2,2} + e^t \frac{a_{2,1}+a_{2,2}}{a_{2,1} a_{1,2} a_{2,2}},
\]
which, up to a toric change of coordinates, corresponds to one of the expressions of \cite{przyjalkowski}.

\bibliography{lagrangian}
\bibliographystyle{alpha}

\end{document}